\documentclass[11pt]{amsart}
\usepackage{amssymb,amsmath,amsthm,newlfont,enumerate}
\usepackage{hyperref}
\hypersetup{
    colorlinks=true,
    linkcolor=black,     
    urlcolor=grey,
    citecolor=blue
}
\theoremstyle{plain}
\newtheorem{theorem}{Theorem}[section]

\newtheorem{proposition}[theorem]{Proposition}

\theoremstyle{definition}

\theoremstyle{remark}

\newcommand{\DD}{\mathbb{D}}

\newcommand{\cD}{\mathcal{D}}

\begin{document}

\date{\today}
\title[Ces\`aro  summability of Taylor series]{Ces\`aro  summability of Taylor series  in weighted Dirichlet spaces}

\author{Javad Mashreghi}
\address{D\'epartement de math\'ematiques et de statistique, Universit\'e Laval,
Qu\'ebec City (Qu\'ebec),  Canada G1V 0A6.}
\email{javad.mashreghi@mat.ulaval.ca}

\author{Pierre-Olivier Paris\'e}
\address{D\'epartement de math\'ematiques et de statistique, Universit\'e Laval,
Qu\'ebec City (Qu\'ebec),  Canada G1V 0A6.}
\email{pierre-olivier.parise.1@ulaval.ca}

\author{Thomas Ransford}
\address{D\'epartement de math\'ematiques et de statistique, Universit\'e Laval,
Qu\'ebec City (Qu\'ebec),  Canada G1V 0A6.}
\email{thomas.ransford@mat.ulaval.ca}

\thanks{JM supported by an NSERC Discovery Grant. 
POP supported by an NSERC Alexander-Graham-Bell Scholarship.
TR supported by grants from NSERC and the Canada Research Chairs program.}

\begin{abstract}
We show that, in every weighted Dirichlet space on the unit disk with superharmonic weight,
the Taylor series of a function in the space is $(C,\alpha)$-summable to the function in the norm of the space, 
provided that $\alpha>1/2$. We further show that the constant $1/2$ is sharp,
in marked contrast with the classical case of the disk algebra.
\end{abstract}



\maketitle

\section{Introduction}\label{S:Intro}
Let $f(z)$ be a formal power series,
say $f(z)=\sum_{k=0}^\infty a_kz^k$.
Many holomorphic function spaces on the unit disk $\DD$ have the property that,
if $f$ belongs to the space, then its Taylor partial sums
\[
s_n[f](z):=\sum_{k=0}^n a_kz^k
\]
converge to $f$ in the norm of the space.
This is the case, for example, 
if the space in question is the Hardy space $H^2$, the Dirichlet space $\cD$ or the Bergman space $A^2$. 
It is also true in all the Hardy spaces $H^p$ for $1<p<\infty$,
even though the proof is not as straightforward as in the other cases.

There are also spaces in which convergence may fail. For instance, a classic example of du Bois-Reymond shows that there exists $f$ in the disk algebra $A(\DD)$ such that $s_n[f]$ does not converge to $f$ in the norm of $A(\DD)$. The same phenomenon can occur in the Hardy space $H^1$. In both of these cases, however, Fej\'er's theorem shows that the Ces\`aro sums
\[
\sigma_n[f](z):=
\frac{1}{n+1}\sum_{k=0}^n s_k[f](z)
=\sum_{k=0}^n\Bigl(1-\frac{k}{n+1}\Bigr)a_kz^k
\]
do converge to $f$ in the norm of the space. 

In the case of the disk algebra, there is a refinement of Fej\'er's theorem 
due to M.~Riesz \cite{Ri09}, who showed that the generalized Ces\`aro means
\begin{equation}\label{E:cesaro}
\sigma_n^\alpha[f](z):=
\binom{n+\alpha}{\alpha}^{-1}\sum_{k=0}^n\binom{n-k+\alpha}{\alpha}a_kz^k
\end{equation}
converge to $f$ in the norm of $A(\DD)$ for each $\alpha>0$. Here,
the binomial coefficients should be interpreted as
\[
\binom{n+\alpha}{\alpha}:=\frac{\Gamma(n+\alpha+1)}{\Gamma(\alpha+1)\Gamma(n+1)},
\]
where $\Gamma$ denotes the gamma function.
An analogous refinement holds in $H^1$.

The fact that $\sigma_n^\alpha[f]$ converges to $f$ in the space is often described by saying that the Taylor series of $f$ is $(C,\alpha)$-summable to $f$ in the space. 
It is well known that $(C,\alpha)$-summability implies $(C,\beta)$-summability if $\alpha<\beta$. 
Thus the $(C,\alpha)$-summability of $f$ for $\alpha>0$ improves Fej\'er's result on the convergence of $\sigma_n[f]$ (namely $(C,1)$-summability) almost to the point of establishing the convergence of $s_n[f]$ itself (namely $(C,0)$-summability).
For background on summability methods, we refer to Hardy's book \cite{Ha49}.

What happens in other spaces?
In this article, we consider the family of weighted Dirichlet spaces with superharmonic weights.  
Dirichlet spaces with harmonic weights were introduced by Richter \cite{Ri91} and further studied by
Richter and Sundberg \cite{RS91}. The generalization to superharmonic weights was treated by Aleman \cite{Al93}.

Let us recall the definition.
Given a positive superharmonic function $\omega$  on $\DD$ 
and a holomorphic function $f$ on $\DD$, we define
\[
\cD_\omega(f):=\int_\DD|f'(z)|^2\,\omega(z)\,dA(z),
\]
where $dA$ denotes normalized area measure on $\DD$.
The \emph{weighted Dirichlet space} $\cD_\omega$ 
is the set of holomorphic $f$ on $\DD$ with $\cD_\omega(f)<\infty$.
Defining
\[
\|f\|_{\cD_\omega}^2:=|f(0)|^2+\cD_\omega(f) 
\qquad(f\in\cD_\omega),
\]
makes $\cD_\omega$ into a Hilbert space containing the polynomials.

It is known that, if $\omega$ is a superharmonic weight and if $f\in\cD_\omega$, 
then $\sigma_n[f]\to f$ in $\cD_\omega$. In particular, polynomials are dense in $\cD_\omega$.
On the other hand, there exist a superharmonic weight $\omega$
and a function $f\in\cD_\omega$ such that $s_n[f]\not\to f$ in $\cD_\omega$.
For proofs of these facts, see \cite[Theorem~1.6]{MR19}.

Thus $\cD_\omega$ behaves a bit like the spaces $A(\DD)$ and $H^1$.
By analogy with what happens in these spaces, 
we might therefore expect Taylor series to be $(C,\alpha)$-summable in $\cD_\omega$ for all $\alpha>0$. 
This turns out not to be the case. We shall establish the following results.

\begin{theorem}\label{T:thm1}
If $\omega$ is a superharmonic weight on $\DD$, 
 if $f\in\cD_\omega$ and if $\alpha>1/2$, then $\sigma_n^\alpha[f]\to f$ in $\cD_\omega$.
\end{theorem}

\begin{theorem}\label{T:thm2}
Let $\omega_1$ be the harmonic weight on $\DD$ defined by $\omega_1(z):=(1-|z|^2)/|1-z|^2$.
Then there exists  $f\in\cD_{\omega_1}$ such that  $\sigma_n^{1/2}[f]\not\to f$ in $\cD_{\omega_1}$.
\end{theorem}

In the terminology of \cite{RS91}, the space $\cD_{\omega_1}$ is a local Dirichlet space.
Theorem~\ref{T:thm2} shows that, even though $\cD_{\omega_1}$ is  a Hilbert space,  
Taylor series in the space actually have worse summability behaviour than in $A(\DD)$ or $H^1$. 

The proofs of Theorems~\ref{T:thm1} and \ref{T:thm2} make use of the theory
of Hadamard multiplication operators of $\cD_\omega$ as developed in \cite{MR19}.
In \S\ref{S:hadamard} we briefly review this theory, before passing to the proofs of the theorems
themselves in \S\ref{S:proofs}. 


\section{Hadamard multiplication operators}\label{S:hadamard}

Given formal power series  $h(z):=\sum_{k=0}^\infty c_kz^k$ and $f(z):=\sum_{k=0}^\infty a_kz^k$,
we define their \emph{Hadamard product} to be the formal power series given by the formula
\[
(h*f)(z):=\sum_{k=0}^\infty c_ka_k z^k.
\]
Obviously, if  $h$  is a polynomial, then $h*f$ is a polynomial too. 
In this case, for each superharmonic weight $\omega$ on $\DD$,
the map $M_h:f\mapsto h*f$ is a bounded linear map from $\cD_\omega$ to itself,
sometimes called a \emph{Hadamard multiplication operator}.
We are interested in estimating its operator norm, $\|M_h:\cD_\omega\to\cD_\omega\|$.

To state our results, we need a little extra notation.
Given a sequence of complex numbers $(c_k)_{k\ge1}$, we write
$T_c$ for the infinite matrix
\begin{equation}\label{E:Tc}
T_c:=
\begin{pmatrix}
c_1 &c_2-c_1 &c_3-c_2 &c_4-c_3 &\dots\\
0 &c_2 &c_3-c_2 &c_4-c_3 &\dots\\
0 &0 &c_3 &c_4-c_3 &\dots\\
0 &0 &0 &c_4 &\dots\\
\vdots &\vdots &\vdots &\vdots &\ddots
\end{pmatrix}.
\end{equation}
If the $(c_k)$ are the  coefficients of a formal power series $h(z)=\sum_{k=0}^\infty c_kz^k$,
then we also write $T_h$ in place of $T_c$. Note that, in this situation, 
the coefficient $c_0$ plays no role.

\begin{theorem}\label{T:mult}
Let $h$ be a polynomial.
\begin{enumerate}[\normalfont (i)]
\item For each superharmonic weight $\omega$ on $\DD$, we have
\[
\|M_h:\cD_\omega\to\cD_\omega\|\le \|T_h:\ell^2\to\ell^2\|.
\]
\item If $\omega_1$ is the harmonic weight on $\DD$ given by $\omega_1(z):=(1-|z|^2)/|1-z|^2$, then
\[
\|M_h:\cD_{\omega_1}\to\cD_{\omega_1}\|= \|T_h:\ell^2\to\ell^2\|.
\]
\end{enumerate}
\end{theorem}

\begin{proof}
Part~(i) is a special case of \cite[Theorem~1.1]{MR19}. Part~(ii) is established in the course of the
proof of the same result, see \cite[p.52]{MR19}.
\end{proof}

To apply Theorem~\ref{T:mult}, 
it is helpful to have at our disposal some explicit estimates for $\|T_c:\ell^2\to\ell^2\|$. 

\begin{theorem}\label{T:Tc}
Let $c:=(c_k)_{k\ge1}$ be a sequence of complex numbers that is eventually zero,
and let $T_c$ be defined by \eqref{E:Tc}.

\begin{enumerate}[{\normalfont(i)}]
\item If $n$ is an integer such that $c_k=0$ for all $k>n$, then
\[
\|T_c:\ell^2\to\ell^2\|^2\le (n+1)\sum_{k=1}^n |c_{k+1}-c_k|^2.
\]
\item For all integers $m,n$ with $1\le m\le n$, we have
\[
\|T_c:\ell^2\to\ell^2\|^2\ge m\sum_{k=m}^n |c_{k+1}-c_k|^2.
\]
\end{enumerate}
\end{theorem}

\begin{proof}
Part (i) was already established in \cite[Theorem~1.2(ii)]{MR19}. 
For part~(ii), we remark that the operator norm of $T_c$ is bounded below by 
the norm of any submatrix, in particular that of
the $m\times (n-m+1)$ submatrix
\[
A:=
\begin{pmatrix}
c_{m+1}-c_m &\dots &c_{n+1}-c_n\\
\vdots &\ddots &\vdots\\
c_{m+1}-c_m &\dots &c_{n+1}-c_n
\end{pmatrix}.
\]
Now $AA^*$ is an $m\times m$ matrix, all of whose entries are the same,
namely $\sum_{k=m}^n|c_{k+1}-c_k|^2$. It follows that
\[
\|T_c\|^2\ge \|A\|^2=\|AA^*\|=m\sum_{k=m}^n|c_{k+1}-c_k|^2.\qedhere
\]
\end{proof}


\section{Proofs of Theorems~\ref{T:thm1} and \ref{T:thm2}}\label{S:proofs}

Instead of using the Ces\`aro means $\sigma_n^\alpha$, defined in \eqref{E:cesaro},
we prefer to work with the so-called discrete Riesz means, defined as follows.
Given a formal power series $f(z):=\sum_{k=0}^\infty a_kz^k$ and $\alpha>0$, we let
\[
\rho_n^\alpha[f](z):=\sum_{k=0}^n \Bigl(1-\frac{k}{n+1}\Bigr)^\alpha a_kz^k.
\]
The following result allows to us pass between Ces\`aro means and Riesz means,
at least when $0<\alpha<1$.

\begin{proposition}\label{P:equiv}
Let $\omega$ be a superharmonic weight on $\DD$, let $f\in\cD_\omega$ and let $0<\alpha<1$.
Then $\sigma_n^\alpha[f]\to f$ in $\cD_\omega$ if and only if $\rho_n^\alpha[f]\to f$ in $\cD_\omega$.
\end{proposition}

To prove this proposition we use a theorem due to M.~Riesz \cite{Ri24}.
Riesz actually proved the result in the scalar case, but
a careful reading of Riesz's proof shows that the theorem easily extends to general Banach spaces.

\begin{theorem}\label{T:equiv}
Let $X$ be a Banach space, 
let $(x_k)_{k\ge0}$ be a sequence in $X$ and let $y$ be an element of $X$.
Then, for each $\alpha\in(0,1)$,
\[
\lim_{n\to\infty}\binom{n+\alpha}{\alpha}^{-1}\sum_{k=0}^n\binom{n-k+\alpha}{\alpha}x_k=y
\iff
\lim_{n\to\infty}\sum_{k=0}^n \Bigl(1-\frac{k}{n+1}\Bigr)^\alpha x_k=y.
\]
\end{theorem}

\begin{proof}[Proof of Proposition~\ref{P:equiv}]
The result follows directly upon applying Theorem~\ref{T:equiv}
with $X:=\cD_\omega$ and $y:=f(z)=\sum_{k\ge0}a_kz^k$ and $x_k:=a_kz^k$.
\end{proof}

\begin{proof}[Proof of Theorem~\ref{T:thm1}]
Let $\omega$ be a superharmonic weight on $\DD$.
To show that Taylor series are $(C,\alpha)$-summable in $\cD_\omega$ for all $\alpha>1/2$,
it suffices to do so for $\alpha\in(\frac{1}{2},1)$. Fix such an $\alpha$.

By Proposition~\ref{P:equiv}, it is enough to show that $\rho_n^\alpha[f]\to f$ in $\cD_\omega$
for all $f\in\cD_\omega$. It is obvious that $\rho_n^\alpha[f]\to f$ if $f$ is a polynomial, and,
as noted in the introduction, polynomials are dense in $\cD_\omega$. Therefore the result will 
follow if we can show that the operator norms of the linear maps
$f\mapsto\rho_n^\alpha[f]:\cD_\omega\to\cD_\omega$
are bounded independently of $n$.

To do this, we identify these maps as a Hadamard multiplication operators.
Indeed, we have $\rho_n^\alpha[f]=M_{h_n}(f)$, where 
\[
h_n(z)=\sum_{k=0}^n \Bigl(1-\frac{k}{n+1}\Bigr)^\alpha z^k.
\]
By Theorem~\ref{T:mult}(i), we have
\[
\|M_{h_n}:\cD_\omega\to\cD_\omega\|\le \|T_{h_n}:\ell^2\to\ell^2\|,
\]
and using Theorem~\ref{T:Tc}(i), we obtain
\begin{align*}
\Bigl\|T_{h_n}:\ell^2\to\ell^2\Bigr\|^2
&\le (n+1)\sum_{k=1}^n\Bigl|\Bigl(1-\frac{k+1}{n+1}\Bigr)^\alpha -\Bigl(1-\frac{k}{n+1}\Bigr)^\alpha \Bigr|^2\\
&=\frac{1}{(n+1)^{2\alpha-1}}\sum_{k=1}^n \Bigl((n+1-k)^\alpha-(n-k)^\alpha\Bigr)^2\\
&=\frac{1}{(n+1)^{2\alpha-1}}\sum_{k=1}^n\Bigl(\int_{n-k}^{n+1-k}\alpha t^{\alpha -1}\,dt\Bigr)^2\\
&\le\frac{1}{(n+1)^{2\alpha-1}}\sum_{k=1}^n\int_{n-k}^{n+1-k}\alpha^2 t^{2\alpha -2}\,dt\\
&=\frac{1}{(n+1)^{2\alpha-1}}\int_{0}^{n}\alpha^2 t^{2\alpha -2}\,dt\\
&\le\frac{\alpha^2}{2\alpha-1}.
\end{align*}

Thus the operator norms of $f\mapsto\rho_n^\alpha[f]:\cD_\omega\to\cD_\omega$ 
are indeed bounded independently of $n$, 
and the proof is complete.
\end{proof}

\begin{proof}[Proof of Theorem~\ref{T:thm2}]
Let $\omega_1(z):=(1-|z|^2)/|1-z|^2$.
By Proposition~\ref{P:equiv}, to show that there exists $f\in\cD_{\omega_1}$ 
whose Taylor series is not $(C,\frac{1}{2})$-summable to $f$,
it is enough to show that there exists $f$ such that $\rho_n^{1/2}[f]\not\to f$ in $\cD_{\omega_1}$.

We   prove that the operator norms of the maps $f\mapsto\rho_n^{1/2}[f]:\cD_{\omega_1}\to\cD_{\omega_1}$ 
tend to infinity as $n\to\infty$. If so, then, by the Banach--Steinhaus theorem, there exists 
$f\in\cD_{\omega_1}$ such that the sequence $\rho_n^{1/2}[f]$ is unbounded in $\cD_{\omega_1}$.
In particular, $\rho_n^{1/2}[f]\not\to f$ in $\cD_{\omega_1}$, as desired.

Once again, to estimate the norm of the map $f\mapsto\rho_n^{1/2}[f]$,
we identify it as a Hadamard multiplication operator, namely
$\rho_n^{1/2}[f]=M_{h_n}(f)$, where 
\[
h_n(z)=\sum_{k=0}^n \Bigl(1-\frac{k}{n+1}\Bigr)^{1/2} z^k.
\]
By Theorem~\ref{T:mult}(ii), we have
\[
\|M_{h_n}:\cD_{\omega_1}\to\cD_{\omega_1}\|= \|T_{h_n}:\ell^2\to\ell^2\|,
\]
and, using Theorem~\ref{T:Tc}(ii), for each $m$ with $1\le m\le n$, we have
\begin{align*}
\Bigl\|T_{h_n}:\ell^2\to\ell^2\Bigr\|^2
&\ge m\sum_{k=m}^n\Bigl|\Bigl(1-\frac{k+1}{n+1}\Bigr)^{1/2} -\Bigl(1-\frac{k}{n+1}\Bigr)^{1/2} \Bigr|^2\\
&=\frac{m}{n+1}\sum_{k=m}^n\Bigl((n+1-k)^{1/2} -(n-k)^{1/2}\Bigr)^2\\
&=\frac{m}{n+1}\sum_{k=m}^n\Bigl(\frac{1}{(n+1-k)^{1/2} +(n-k)^{1/2}}\Bigr)^2\\
&\ge \frac{m}{4(n+1)}\sum_{k=m}^n\frac{1}{n+1-k}\\
&\ge  \frac{m}{4(n+1)}\log(n+2-m).
\end{align*}
In particular, taking $m:=\lceil (n+1)/2\rceil$, we obtain
\[
\Bigl\|T_{h_n}:\ell^2\to\ell^2\Bigr\|^2
\ge \frac{1}{8}\log\Bigl(\frac{n+1}{2}\Bigr),
\]
which tends to infinity with $n$.

Thus the operator norms of $f\mapsto\rho_n^{1/2}[f]:\cD_\omega\to\cD_\omega$ 
tend to infinity with~$n$, as claimed, and the proof is complete.
\end{proof}

	

\bibliographystyle{plain}
\bibliography{biblio}

\end{document}